\documentclass[10pt]{article}
\usepackage[all]{xy}
\usepackage{amsfonts,amsmath,oldgerm,amssymb,amscd}
\newcommand{\ra}{\rightarrow}

\newcommand{\surj}{\ra\!\!\!\ra}	
\newcommand{\ol}{\overline}

\newtheorem{theorem}{Theorem}[section]
\newtheorem{proposition}[theorem]{Proposition}
\newtheorem{lemma}[theorem]{Lemma}

\newtheorem{corollary}[theorem]{Corollary}
\newtheorem{question}[theorem]{Question}

\newcommand{\ga}{\alpha}	\newcommand{\gb}{\beta}

\newcommand{\gj}{\blacksquare}	\newcommand{\gl}{\lambda}
\newcommand{\gm}{\gamma}	
		\newcommand{\gs}{\sigma}

	\newcommand{\BF}{\mbox{$\mathbb F$}}

	\newcommand{\BN}{\mbox{$\mathbb N$}}
	
	\newcommand{\BR}{\mbox{$\mathbb R$}}

	\newcommand{\BZ}{\mbox{$\mathbb Z$}}

\newcommand{\CC}{\mbox{$\mathcal C$}}

\newcommand{\mm}{\mbox{$\mathfrak m$}}	
	\newcommand{\p}{\mbox{$\mathfrak p$}}
\newcommand{\mq}{\mbox{$\mathfrak q$}}

\newcommand{\ot}{\mbox{\,$\otimes$\,}}	\newcommand{\op}{\mbox{$\oplus$}}
\newcommand{\Spec}{\mbox{\rm Spec\,}}	\newcommand{\hh}{\mbox{\rm ht\,}}
	
\newcommand{\Aut}{\mbox{\rm Aut\,}}	
\newcommand{\Hom}{\mbox{\rm Hom\,}}	

\newcommand{\Um}{\mbox{\rm Um}}		\newcommand{\SL}{\mbox{\rm SL}}
\newcommand{\GL}{\mbox{\rm GL}}		
	\newcommand{\E}{\mbox{\rm E}}

\begin{document}
\begin{center}
{\Large \bf Serre Dimension of Monoid Algebras}\\\vspace{.2in} {\large
  Manoj K. Keshari and Husney Parvez Sarwar }\\
\vspace{.1in}
{\small
Department of Mathematics, IIT Bombay, Powai, Mumbai - 400076, India;\;\\
    keshari@math.iitb.ac.in; mathparvez@gmail.com}
\end{center}
%\maketitle

\begin{abstract}
Let $R$ be a commutative Noetherian ring of dimension $d$, $M$ a
commutative cancellative torsion-free monoid of rank $r$ and $P$ a
finitely generated projective $R[M]$-module of rank $t$.
  
$(1)$ Assume $M$ is $\Phi$-simplicial seminormal.  $(i)$ If
  $M\in \CC(\Phi)$, then {\it Serre dim} $R[M]\leq
  d$.  $(ii)$ If $r\leq 3$, then {\it Serre dim} $R[int(M)]\leq d$.
  
 $(2)$ If $M\subset \BZ_+^2$ is a
normal monoid of rank $2$, then {\it Serre dim} $R[M]\leq d$.

$(3)$ Assume $M$ is $c$-divisible, $d=1$ and $t\geq 3$. Then
  $P\cong \wedge^t P\op R[M]^{t-1}$.
 
 $(4)$ Assume $R$ is a uni-branched affine algebra over an
  algebraically closed field and $d=1$. 
Then $P\cong \wedge^t P\op R[M]^{t-1}$.
 \end{abstract}

\section{Introduction}

{\it Throughout rings are commutative Noetherian with $1$; projective
  modules are finitely generated and of constant rank; monoids are 
commutative cancellative torsion-free;
  $\BZ_+$ denote the additive monoid of non-negative integers.}

Let $A$ be a ring and $P$ a projective $A$-module. An element $p\in P$
is called {\it unimodular}, if there exists $\phi\in \Hom(P,A) $ such
that $\phi(p)=1$. We say Serre dimension of $A$ (denoted as {\it Serre
  dim $A$}) is $\leq t$, if every projective $A$-module of rank $\geq
t+1$ has a unimodular element.  Serre dimension of $A$ measures the
surjective stabilization of the Grothendieck group $K_0(A)$. Serre's
problem on the freeness of projective $k[X_1,\ldots,X_n]$-modules, $k$
a field, is equivalent to {\it Serre dim} $k[X_1,\ldots,X_n]=0$.

After the solution of Serre's problem by Quillen \cite{Qu} and Suslin
\cite{Su76}, many people worked on surjective stabilization of 
polynomial extension of a ring.  Serre \cite{Serre} proved {\it Serre
  dim} $A \leq \dim A$, Plumstead \cite{P} proved {\it Serre dim}
$A[X]\leq \dim A$, Bhatwadekar-Roy \cite{BhR} proved {\it Serre dim}
$A[X_1,\ldots,X_n]\leq \dim A$ and Bhatwadekar-Lindel-Rao \cite{BLR}
proved {\it Serre dim} $A[X_1,\ldots,X_n,Y_1^{\pm 1},\ldots,Y_m^{\pm
    1}]\leq \dim A$.

Anderson conjectured an analogue of Quillen-Suslin theorem for monoid
algebras over a field which was answered by Gubeladze
\cite{Gu} (see \ref{G1}) as follows.

\begin{theorem}\label{G1}
Let $k$ be a field and $M$ a monoid. Then $M$ is seminormal if and
only if all projective $k[M]$-modules are free.
\end{theorem}

Gubeladze \cite{G1} asked the following 

\begin{question}\label{3q1}
Let $M\subset \BZ_+^r$ be a monoid of rank $r$ with $M\subset \BZ_+^r$ an
integral extension. Let $R$ be a ring of dimension $d$.
 Is {\it Serre dim} $R[M]\leq d$?
  \end{question}  
 
We answer Question \ref{3q1} for some class of monoids. Recall that a
finitely generated monoid $M$ of rank $r$ is called {\it $\Phi$-simplicial}
if $M$ can be embedded in $\BZ_+^r$ and the extension $M\subset
\BZ_+^r$ is integral (see \cite{G2}). A $\Phi$-simplicial monoid is
commutative, cancellative and torsion-free.

\begin{define}\label{3d11}
 Let $\CC(\Phi)$ denote the class of seminormal $\Phi$-simplicial
 monoids $M\subset \BZ_+^{r}$ of rank $r$ such that if
 $\BZ_+^r=\{t_1^{s_1}\ldots t_{r}^{s_r}|\, s_i\geq 0\}$, then for
 $1\leq m \leq r$, $M_m=M\cap \{t_1^{s_1}\ldots t_{m}^{s_m}|\, s_i\geq
 0\}$ satisfies the following properties: Given a positive integer
 $c$, there exist integers $c_i > c$ for $i=1,\ldots, m-1$ such that
 for any ring $R$, the automorphism $\eta \in Aut_{R[t_m]}(
 R[t_1,\ldots,t_{m}])$ defined by $\eta(t_i)=t_i+t_m^{c_i}$ for $
 i=1,\ldots, m-1$, restricts to an $R$-automorphism of $R[M_m]$.  It
 is easy to see that $M_m\in \CC(\Phi)$ and rank $M_m=m$ for $1\leq
 m\leq r$.
\end{define} \medskip

The following result (\ref{3t5}, \ref{3t7}) answers
Question \ref{3q1} for monoids in 
$\CC(\Phi)$.

\begin{theorem}\label{3t1} 
Let $M\subset \BZ_+^r$ be a seminormal $\Phi$-simplicial monoid of
rank $r$ and $R$ a ring of dimension $d$.

$(1)$ If $M\in \CC(\Phi)$, then {\it Serre dim} $R[M]\leq d$. 

$(2)$ If $r\leq 3$, then {\it Serre dim} $R[int(M)]\leq d$, where
$int(M)=int(\BR_+M)\cap \BZ_+^3$ and $int(\BR_+M)$ is the interior of
the cone $\BR_+M\subset \BR^3$ with respect to Euclidean topology.
\end{theorem}

The following result (\ref{3c11}) follows from (\ref{3t1}(1)).
When $R$ is a field, this result is due to Anderson \cite{An78}.

\begin{theorem}
  Let $R$ be a ring of dimension $d$ and $M\subset \BZ_+^2$ a
normal monoid of rank $2$. Then {\it Serre dim} $R[M]\leq d$.
\end{theorem}

The next result answers Question \ref{3q1} partially for
$1$-dimensional rings (see \ref{3t15}, \ref{3t16}). The proof uses the
techniques of Kang \cite{K}, Roy \cite{Roy82} and Gubeladze's
\cite{Gu90}. Let us recall two definitions. (i) A monoid $M$ is
called {\it $c$-divisible}, where $c>1$ is an integer, if $cX=m$ has a
solution in $M$ for all $m\in M$. All $c$-divisible monoids are
seminormal. (ii) Let $R$ be a ring, $\ol R$ the integral closure of
$R$ and $C$ the conductor ideal of $R\subset \ol R$. Then
$R$ is called {\it uni-branched} if for any
$\p \in \Spec R$ containing $C$, there is a unique $\mq \in \Spec \ol
R$ such that $\mq \cap R=\p$.

\begin{theorem}\label{1a}
Let $R$ be a ring of dimension $1$, $M$ a monoid and
$P$ a projective $R[M]$-module of rank $r$.

(i) If $M$ is $c$-divisible and $r\geq 3$, then $P\cong
\wedge^rP\op R[M]^{r-1}$.

(ii) If $R$ is a uni-branched affine algebra over an algebraically
closed field, then $P\cong \wedge^rP\op R[M]^{r-1}$.
\end{theorem}

If $R$ is a $1$-dimensional anodal ring with finite seminormalization,
then (\ref{1a}(ii)) is due to Sarwar (\cite{Sarwar}, Theorem
1.2). Note that if $k$ is an algebraically closed field of
characteristic $2$, then node $k[X,Y]/(X^2-Y^2-Y^3)$ is not anodal but
is uni-branched, by Kang (\cite{K}, Example 2).

At the end, we give some applications 
to minimum number of generators of projective modules.

%%%%%%%%%%%%%%%%%%%%%%%%%%%%%%%%%%%%%%%%%%%%%%%%%%%%
\section{Preliminaries}
 
Let $A$ be a ring and $Q$ an $A$-module. We say $p\in Q$ is unimodular
if the order ideal $O_Q(p)=\{\phi(p) \,|\,\, \phi \in \Hom(Q,A)\}$ equals
$A$.  The set of all unimodular elements in $Q$ is denoted by
$\Um(Q)$. We write $\E_n(A)$ for the group generated by set of all
$n\times n$ elementary matrices over $A$ and $\Um_n(A)$ for
$\Um(A^n)$. We denote by $\Aut_A(Q )$, the group of all $A$-automorphisms
of $Q$.

For an ideal $J$ of $A$, we denote by $\E(A \oplus Q, J)$, the subgroup
of $\Aut_A (A\oplus Q )$ generated by all the automorphisms
$\Delta_{a\phi}= \bigl( \begin{smallmatrix} 1 & a\phi \\ 0&id_Q
\end{smallmatrix} \bigr)$ and $\Gamma_{q}= \bigl( \begin{smallmatrix}
1 & 0 \\ q&id_Q
\end{smallmatrix} \bigr)$ with $a\in J$, $\phi \in Q^*$ and $q\in Q$. 
Further, we shall write $\E(A\oplus Q)$ for $\E(A\oplus Q, A)$.
We denote by $\Um(A\oplus Q, J)$ the set of all $(a,q) \in \Um(A\oplus
Q )$ with $a \in 1 + J$ and $q \in JQ$.

We state some results of Lindel \cite{L95} for later use.
 
\begin{proposition}
(Lindel \cite{L95}, 1.1) Let $A$ be a ring and $Q$ an $A$-module. Let $Q_s$
  be free of rank $r$ for some $s\in A$. Then there exist
  $p_1,\ldots,p_r \in Q,\,\, \phi_1,\ldots,\phi_r\in Q^*$ and $t\geq
  1$ such that following holds:

$(i)$ $0:_A s'A=0:_A{s'}^{2}A$, where $s'=s^t$.
 
$(ii)$ $s'Q \subset F$ and $s'Q^*\subset G$, where $F=\sum_{i=1}^r
 Ap_i\subset Q$ and $G=\sum _{i=1}^r A\phi_i \subset Q^*$.
  
$(iii)$ the matrix $(\phi_i(p_j))_{1\leq i,j\leq r}=diagonal\,(s',\ldots,s')$. We
 say $F$ and $G$ are $s'$-dual submodules of $Q$ and $Q^*$
 respectively.
 \end{proposition}

\begin{proposition}\label{3p1}
(Lindel \cite{L95}, 1.2, 1.3) Let $A$ be a ring and $Q$ an $A$-module. Assume
  $Q_s$ is free of rank $r$ for some $s\in A$. Let $F$ and $G$ be
  $s$-dual submodules of $Q$ and $Q^*$ respectively. Then

 $(i)$ for $p\in Q$,
  there exists $q\in F$ such that $\hh (O_Q(p+sq)A_s)\geq r$.

  $(ii)$ If $Q$ is projective $A$-module and $\ol p \in \Um(Q/sQ)$, then there
  exists $q\in F$ such that $\hh (O_Q(p+sq))\geq r$.
\end{proposition}

\begin{proposition}\label{3p3}
(Lindel \cite{L95}, 1.6) Let $Q$ be a module over a positively graded ring
  $A=\oplus_{i\geq 0}A_i$ and $Q_s$ be free for some $s\in R=A_0$. Let
  $T\subseteq A$ be a multiplicatively closed set of homogeneous
  elements. Let $p\in Q$ be such that $p_{T(1+sR)}\in \Um
  (Q_{T(1+sR)})$ and $s\in rad(O_Q(p)+A_+)$, where $A_+=\oplus_{i\geq
    1}A_i$. Then there exists $p'\in p+sA_+Q$ such that $p'_T\in
  \Um(Q_T)$.
\end{proposition}

\begin{proposition}\label{3p4}
(Lindel \cite{L95}, 1.8)
 Under the assumptions of (\ref{3p3}),
 let $p\in Q$ be such that $O_Q(p)+sA_+=A$ and $A/O_Q(p)$
 is an integral extension of $R/(R\cap O_Q(p))$. Then there exists
 $p'\in$ $\Um(Q)$ with $p'-p\in sA_+Q$.
\end{proposition}

The following result is due to Amit Roy (\cite{Roy82}, Proposition 3.4).

\begin{proposition}\label{3p8}
Let $A,B$ be two rings with $f:A\ra B$ a ring homomorphism. 
Let $s\in A$ be non-zerodivisor such that $f(s)$ is a non-zerodivisor in $B$.
Assume that we have the following cartesian square.
\[
\xymatrix{
A \ar@{->}[r]^{f}
     \ar@{->}[d]
& B 
     \ar@{->}[d]
\\
A_s \ar@{->}[r]^{f_s}
     & B_{f(s)}      
}
\]
Further assume that $\SL_r(B_{f(s)})=\E_r(B_{f(s)})$ for some $r>0$. Let $P$ and $Q$ be two projective
$A$-modules of rank $r$ such that $(i)$ $\wedge^rP\cong \wedge^rQ$, $(ii)$ $P_s$ and $Q_s$ are free over $A_s$,
$(iii)$ $P\ot_{A}B\cong Q\ot_A B $ and $Q\ot_A B$ has a unimodular element.
Then $P\cong Q$.

\end{proposition}

\begin{define}\label{def}
(see \cite{G2}, Section 6) Let $R$ be a ring and $M$ a
  $\Phi$-simplicial monoid of rank $r$. Fix an integral extension
  $M\hookrightarrow \BZ^r_+$.  Let $\{t_1,\ldots,t_r\}$ be a free
  basis of $\BZ^r_+$. Then $M$ can be thought of as a monoid
  consisting of monomials in $t_1,\ldots,t_r$. 

For $x=t_1^{a_1}\ldots t_r^{a_r}$ and $y=t_1^{b_1}\ldots t_r^{b_r}$ in
$\BZ^r_+$, define $x$ is {\it lower} than $y$ if $a_i < b_i$ for
some $i$ and $a_j=b_j$ for $j > i$. In particular, $t_i$ is lower than
$t_j$ if and only if $i<j$. 

For $f\in R[M]$, define the {\it
  highest member} $H(f)$ of $f$ as $am$, where
$f=am+a_1m_1+\ldots+a_km_k$ with $m,m_i\in M$, $a\in R\setminus \{0\},
a_i\in R$ and each $m_i$ is strictly lower than $m$ for $1\leq i\leq
k$.

 An element $f\in R[\BZ^{r}_+]$ is
  called {\it monic} if $H(f)=at^s_r$, where $a\in R$ is a unit and $s>0$.
  An element $f\in R[M]$ is said to be {\it monic} if $f$ is monic in
  $R[\BZ^r_+]$ via the embedding $R[M]\hookrightarrow R[\BZ^r_+]$. 

Define $M_0$ to be the submonoid  $\{ t_1^{s_1}\ldots
  t_{r-1}^{s_{r-1}}|\, s_i\geq 0\}\cap M$ of  $M$.  Clearly $M_0$ is finitely
  generated as $M$ is finitely generated. Also
  $M_0\hookrightarrow \BZ^{r-1}_+$ is integral. Hence
  $M_0$ is $\Phi$-simplicial. Further, if $M$ is
  seminormal, then $M_0$ is seminormal.

 Grade $R[M]$ as $R[M]=R[M_0]\oplus A_1\oplus A_2\oplus \ldots$, where
 $A_i$ is the $R[M_0]$-module generated by the monomials
 $t_1^{s_1}\ldots t_{r-1}^{s_{r-1}}t_r^i \in M$.  For an ideal $I$ in
 $R[M]$, define its leading coefficient ideal $\gl(I)$ as $\{a\in R
 \,|\,\exists f\in I$ with $H(f)=am$ for some $m\in M\}$.  $\hfill
 \gj$
\end{define}

\begin{lemma}\label{3l2}
(\cite{G2}, Lemma 6.5)
 Let $R$ be a ring and $M\subset \BZ_+^r$ a $\Phi$-simplicial
 monoid. If $I\subseteq R[M]$ is an ideal, then $\hh(\gl(I))\geq
 \hh (I)$, where $\gl(I)$ is defined in $(\ref{def})$.
\end{lemma}

%$$$$$$$$$$$$$$$$$$$$$$$$$$$$$$$$$$$$$$$$$$$$$$$$$$$$$

%%%%%%%%%%%%%%%%%%%%%%%%%%%%%%
\section{Main Theorem}

This section contains main results stated in the introduction. We also
give some examples of monoids in $C(\Phi)$.

\subsection{Over $C(\Phi)$ class of monoids}
%We begin with a lemma.

\begin{lemma}\label{3l7}
 Let $R$ be a ring and $M\subset \BZ_+^r$ a monoid in $\CC(\Phi)$ of rank $r$.
  Let $f\in R[M]\subset R[\BZ_+^r]=
 R[t_1,\ldots,t_r]$ with $H(f)=ut_1^{s_1}\ldots t_r^{s_r}$ for some
 unit $u\in R$. Then there exist $\eta \in Aut_R(R[M])$
 such that $\eta(f)$ is a monic polynomial in $t_r$.
\end{lemma}

\begin{proof}
Since $M\in C(\Phi)$, we can 
 choose positive integers $c_1,\ldots,c_{r-1}$ such that the automorphism
$\eta\in
 Aut_{R[t_r]}R[t_1\ldots,t_r]$ defined by $\eta(t_i)=t_i+t_r^{c_i}$ for
 $i=1,\ldots,r-1$, restricts to an automorphism of $R[M]$ and such that
$\eta(f)$ is a monic polynomial in $t_r$. $\hfill \gj$
\end{proof}

\begin{lemma}\label{3l4}
Let $R$ be a ring of dimension $d$ and $M\subset \BZ_+^r$ a monoid in $\CC(\Phi)$ of rank $r$.
Let $P$ be a projective $R[M]$-module of rank $>d$. Write
$R[M]=R[M_0]\oplus A_1\oplus A_2\ldots$, as defined in
$(\ref{def})$ and $A_+=A_1\oplus A_2\oplus \ldots$ an ideal of
$R[M]$.  Assume that $P_s$ is free for some $s\in R$ and $P/sA_+P$ has
a unimodular element. Then the natural map $\Um(P)\ra \Um(P/sA_+P)$ is
surjective. In particular, $P$ has a unimodular element.
\end{lemma}

\begin{proof}
Write $A=R[M]$. Since every unimodular element of $P/sA_+P$ can be
lifted to a unimodular element of $P_{1+sA_+}$, if $s$ is nilpotent,
then elements of $1+sA_+$ are units in $A$ and we are done.
Therefore, assume that $s$ is not nilpotent.

Let $p\in P$ be such that $\ol p\in \Um(P/sA_+P)$.  Then
$O_P(p)+sA_+=A$. Hence $O_P(p)$ contains an element of
$1+sA_+$. Choose $g\in A_+$ such that $1+sg\in O_P(p)$. Applying
(\ref{3p1}) with $sg$ in place of $s$, we get $q\in F\subset P$ such
that $\hh (O_P(p+sgq)) > d$.  Since $p+sgq$ is a lift of $\ol p$,
replacing $p$ by $p+sgq$, we may assume that $\hh (O_P(p))>d$. By
(\ref{3l2}), we get $\hh(\gl(O_P(p)))\geq \hh(O_P(p)) >
d$. Since $\gl(O_P(p))$ is an ideal of $R$, we get $1\in \gl(O_P(p))$.
Hence there exists $f\in O_P(p)$ such that the coefficient of
$H(f)$ (highest member of $f$) is a unit.
  
Suppose $H(f)=ut_1^{s_1}\ldots t_r^{s_r}$ with $u$ a unit in
$R$. Since $M\in\CC(\Phi)$, by (\ref{3l7}), there exists $\ga \in
\Aut_R(R[M])$ such that $\ga(f)$ is monic in $t_r$. Thus we may assume
that $O_P(p)$ contains a monic polynomial in $t_r$. Hence $A/O_P(p)$
is an integral extension of $R[M_0]/(O_P(p)\cap R[M_0])$ and $\ol p\in
\Um(P/sA_+P)$. By (\ref{3p4}), there exists $p'\in \Um(P)$ such that
$p'-p \in sA_+P$. This means $p'\in \Um(P)$ is a lift of $\ol p$. This
proves the result.  $\hfill \gj$
\end{proof}

\begin{remark}
In (\ref{3l4}), we do not need the monoid $M$ to be seminormal.
\end{remark}

The next result proves $(\ref{3t1}(1))$.

\begin{theorem}\label{3t5}
 Let $R$ be a ring of dimension $d$ and $M$ a monoid in $\CC(\Phi)$ of
 rank $r$. If $P$ is a projective $R[M]$-module of rank $r'\geq d+1$,
 then $P$ has a unimodular element. In other words, {\it Serre dim}
 $R[M]\leq d$.
\end{theorem}

\begin{proof}
 We can assume that the ring is reduced with connected spectrum.  If
 $d=0$, then $R$ is a field. Since $M$ is seminormal,
 projective $R[M]$-modules are free, by
 (\ref{G1}). If $r=0$,
 then $M=0$ and we are done by Serre \cite{Serre}. 
 Assume $d,r\geq 1$ and use induction on $d$ and $r$
 simultaneously.

If $S$ is the set of all non-zerodivisor of $R$, then $\dim S^{-1}R=0$
and so $S^{-1}P$ is free $S^{-1}R[M]$-module ($d=0$ case). Choose
$s\in S$ such that $P_s$ is free.  Consider the ring
$R[M]/(sR[M])=(R/sR)[M]$. Since $\dim R/sR = d-1$, by induction on
$d$, $\Um(P/sP)$ is non-empty.

Write $R[M]=R[M_0]\oplus A_1\oplus A_2\ldots$, as defined in (\ref{def})
and $A_+=A_1\oplus A_2\oplus \ldots$ an ideal of $R[M]$.
Note that $M_0\in \CC(\Phi)$ and rank $M_0=r-1$.
Since $R[M]/A_+=R[M_0]$,
by induction on $r$,
$\Um(P/A_+P)$ is non-empty.  Write $A=R[M]$ and consider the following
fiber product diagram
 \[
\xymatrix{
A/(sA\cap A_+) \ar@{->}[r]
     \ar@{->}[d]
&A/sA 
     \ar@{->}[d]
\\
A/A_+ \ar@{->}[r]
     &A/(s,A_+)      
}
\]
If $B=R/sR$, then $A/(s,A_+)=B[M_0]$.
 Let $u\in \Um(P/A_+P)$ and $v\in \Um(P/sP)$. Let $\ol {u}$ and
 $\ol {v}$ denote the images of $u$ and $v$ in $P/(s,A_+)P$.  
Write $P/(s,A_+)P=B[M_0] \oplus P_0$, where $P_0$ is some
projective $B[M_0]$-module of rank $=r'-1$. Note that
$dim(B)=d-1$ and $\ol {u}, \ol {v}$ are two unimodular elements in
$B[M_0] \oplus P_0$. 

{\it Case 1.} Assume $rank(P_0)\geq$ max $\{2,d\}$. Then by 
(\cite{DK}, Theorem 4.5), there exists $\gs \in \E(B[M_0]\oplus P_0)$ such
that $\gs(\ol {u})=\ol {v}$. Lift $\gs$ to an element $\gs_1 \in
\E(P/A_+P)$ and write $\gs_1(u)=u_1\in \Um(P/A_+P)$. Then
images of $u_1$ and $v$ are same in $P/(s,A_+)P$. Patching $u_1$ and
$v$ over $P/(s,A_+)P$ in the above fiber product diagram, we get an element
$p \in \Um(P/(sA\cap A_+)P)$.
  
  Note $sA\cap A_+=sA_+$. We have $P_s$ is free
  and $P/sA_+P$ has a unimodular element.  Use (\ref{3l4}), to
  conclude that $P$ has a unimodular element.
   
{\it Case 2.}  Now we consider the remaining case, namely  $d=1$ and
$rank(P)=2$. Since $B=R/sR$ is $0$ dimensional, 
projective modules over $B[M_0]$ and $B[M]$ are free,
by (\ref{G1}). In particular, $P/sP$ and $P/(s,A_+)P$ are
free modules of rank $2$ over the rings $B[M]$ and $B[M_0]$
respectively.  Consider the same fiber product diagram as above.

 Since any two unimodular elements in $\Um_2(B[M_0])$ are connected
 by an element of $\GL_2(B[M_0])$.  Further $B[M_0]$ is a subring of
 $B[M]=A/sA$. Hence the natural map $\GL_2(B[M]) \ra \GL_2(B[M_0])$ is surjective. 
 Hence any automorphism of $P/(s,A_+)P$ can be lifted to
 an automorphism of $P/sP$. By same argument as above, patching
 unimodular elements of $P/sP$ and $P/A_+P$, we get a unimodular
 element in $P/(sA\cap A_+)P$.  Since $sA\cap A_+=sA_+$ and $P/sA_+P$ has
 a unimodular element, by (\ref{3l4}), $P$ has a unimodular element.
This completes the proof.
$\hfill \gj$
\end{proof}

\begin{example}\label{3r5}
 $(1)$ If $M$ is a $\Phi$-simplicial normal monoid of rank $2$, then
  $M\in\CC(\Phi)$. To see this, by (\cite{G2}, Lemma 1.3), $M\cong
  (\ga_1,\ga_2)\cap \BZ_+^2$, where $\ga_1=(a,b)$ and $\ga_2=(0,c)$
  and $(\ga_1,\ga_2)$ is the group generated by $\ga_1$ and $\ga_2$.
  It is easy to see that $M\cong ((1,a_1),(0,a_2))\cap \BZ_+^2$, where
  gcd$(b,c)=g$ and $a_1=b/g,\,a_2=c/g$. Hence $M\in \CC(\Phi)$.

$(2)$ If $M\subset \BZ^2_+$ is a finitely generated rank $2$ normal
  monoid, then it is easy to see that $M$ is $\Phi$-simplicial. Hence
  $M\in \CC(\Phi)$ by (1).

$(3)$ If $M$ is a rank $3$ normal quasi-truncated or truncated monoid
  (see \cite{G2}, Definition 5.1), then $M \in \CC(\Phi)$. To see
  this, by (\cite{G2}, Lemma 6.6), $M$ satisfies properties of $(\ref{3d11})$.
  Further, $M_0$ is a $\Phi$-simplicial normal monoid of
  rank $2$. By (1), $M_0\in\CC(\Phi)$.
$\hfill \gj$

%$(4)$ The monoid $M\subset \BZ_+^3$ generated by $\{x^2,y^2,z^2, xy,
 % xz,yz\}$ is in $\CC(\Phi)$. The submonoid $N$ of $M$
  %generated by $\{x^2,y^2,z^2,xz,yz\}$ is
  %also in $\CC(\Phi)$.  $\hfill \gj$
\end{example}
 
 \begin{corollary}\label{3c11}
  Let $R$ be a ring of dimension $d$ and $M\subset \BZ_+^2$ a
  normal monoid of rank $2$. Then {\it Serre dim} $R[M]
\leq d$.
 \end{corollary}
 
\begin{proof}
If $M$ is finitely generated, then result follows from (\ref{3r5}(2))
and (\ref{3t5}).
 
If $M$ is not finitely generated, then write $M$ as a filtered union
of finitely generated submonoids, say $M=\cup_{\gl\in I}M_{\gl}$.
Since $M$ is normal, the integral closure $\ol M_{\gl}$ of $M_{\gl}$
is contained in $M$.  Hence $M=\cup_{\gl\in I}\ol M_{\gl}$. By
(\cite{B-G}, Proposition 2.22), $\ol M_{\gl}$ is finitely generated.
If $P$ is a projective $R[M]$-module, then $P$ is defined over $R[\ol
  M_{\gl}]$ for some $\gl \in I$ as $P$ is finitely generated. Now the
result follows from (\ref{3r5}(2)) and (\ref{3t5}).  $\hfill \gj$
 \end{proof}

The following result follows from (\ref{3r5}(3)) and  (\ref{3t5}).

\begin{corollary}\label{3c1}
Let $R$ be a ring of dimension $d$ and $M$ a truncated or normal quasi-truncated
 monoid of rank $\leq3$. Then {\it Serre dim} $R[M]\leq d$.
\end{corollary}

Now we prove (\ref{3t1}(2)).

\begin{proposition}\label{3t7}
 Let $R$ be a ring of dimension $d$ and $M$ a $\Phi$-simplicial
 seminormal monoid of rank $\leq 3$.  Then {\it Serre dim}
 $R[int(M)]\leq d$.
\end{proposition}

\begin{proof}
Recall that $int(M)=int(\BR_+M)\cap \BZ_+^3$.
Let $P$ be a projective $R[int(M)]$-module of rank $\geq d+1$.
Since $M$ is seminormal, by (\cite{B-G}, Proposition 2.40),
$int(M)=int(\ol M)$, where $\ol M$ is the normalization of $M$. Since
normalization of a finitely generated monoid is finitely generated
(see \cite{B-G}, Proposition 2.22), $\ol M$ is a $\Phi$-simplicial
normal monoid.  By (\cite{G2}, Theorem 3.1), $int(M)=int(\ol M)$ is a
filtered union of truncated (normal) monoids (see \cite{G2}, Definition 2.2). 
Since $P$ is finitely generated,
we get $P$ is defined over $R[N]$, where $N\subset int(M)$ is a truncated
monoid.  By (\ref{3c1}), {\it Serre dim} $R[N]\leq d$. Hence $P$ has a
unimodular element.  Therefore {\it Serre dim} $R[int(M)]\leq d$.
$\hfill \gj$
\end{proof}
\medskip

\noindent{\bf Assumptions:} In the following examples,
 $R$ is a ring of dimension $d$,  
Monoid operations are written
multiplicatively and $K(M)$ denotes the group of
fractions of monoid $M$.

\begin{example}\label{3e4}
 For $n>0$, consider the monoid $M\subset \BZ_+^r$ generated by
 $\{t_1^{i_1}t_2^{i_2}\ldots t_r^{i_r} |\sum i_j=n\}$.
 Then $M$ is a $\Phi$-simplicial normal
 monoid. For integers $c_i=nk_i+1$, $k_i>0$ and
 $i=1,\ldots,r-1$, consider  $\eta \in \Aut_{R[t_r]}(R[t_1,\ldots, t_r])$ 
 defined by $t_i\mapsto t_i+t_r^{c_i}$
 for $i=1,\ldots,r-1$. 
  
 A typical monomial in the expansion of $\eta(t_1^{i_1}\ldots
 t_{r-1}^{i_{r-1}}t_r^{i_r})=(t_1+t_r^{c_1})^{i_1}\ldots
 (t_{r-1}+t_r^{c_{r-1}})^{i_{r-1}}t_r^{i_r}$ will look like
 $(t_1^{i_1-l_1}t_r^{c_1l_1})\ldots
 (t_{r-1}^{i_{r-1}-l_{r-1}}t_r^{c_{r-1}l_{r-1}})t_r^{i_r}=(t_1^{i_1-l_1}\ldots
 t_{r-1}^{i_{r-1}-l_{r-1}}t_r^{l_1+\ldots+l_{r-1}+i_r})t_r^{n(k_1l_1+\ldots+k_{r-1}l_{r-1})}$
 which belong to $M$.  So $\eta(R[M])\subset R[M]$.  Similarly,
 $\eta^{-1}(R[M])\subset R[M]$. Hence $\eta$ restricts to an
 $R$-automorphism of $R[M]$. Therefore $\eta$ satisfies the property
 of (\ref{3d11}) for $M$. It is easy to see that $M_m=M\cap
 \{t_1^{s_1}\ldots t_m^{s_m}|s_i\geq 0\}$ for $1\leq m \leq r-1$ also
 satisfy this property.  Hence $M\in \CC(\Phi)$. By (\ref{3t5}), {\it
   Serre dim} $R[M]\leq d$.  $\hfill \gj$
\end{example}

\begin{example}\label{3e2}
 Let $M$ be a $\Phi$-simplicial monoid generated by monomials
$t_1^2,t_2^2,t_3^2,t_1t_3,t_2t_3.$  For
 integers $c_j=2k_j-1$ with $k_j>1$, consider the automorphism
 $\eta\in \Aut_{R[t_3]}(R[t_1,t_2,t_3])$ defined by $t_j\mapsto
 t_j+t_3^{c_j}$ for $j=1,2$. 
Then it is easy to see that $\eta$ restricts to an automorphism of $R[M]$.

We claim that $M$ is seminormal
but not normal.  For this, let
$$z=(t_3^2)^{-1}(t_1t_3)(t_2t_3)=t_1t_2\in K(M)\setminus M,~ \rm{but} ~ z^2
\in M,$$ showing that $M$ is not normal.  For seminormality, let
$$z=(t_1^2)^{\ga_1}(t_2^2)^{\ga_2}(t_3^2)^{\ga_3}
(t_1t_3)^{\ga_{4}}(t_{2}t_3)^{\ga_5}\in K(M) ~ \text{with} ~ \ga_i\in \BZ ~ \rm{and} ~
z^2,z^3\in M.$$  We may assume that $0\leq \ga_{4},\ga_{5}\leq 1$. Now
$z^2\in M \Rightarrow \ga_1,\ga_2\geq 0$ and $2\ga_3+\ga_4 +\ga_5\geq
0$. If $\ga_3< 0$, then $\ga_4=\ga_5=1$ and $\ga_3=-1$.  In this case,
$z^3=(t_1^{2\ga_1+1}t_2^{2\ga_2+1})^3\notin M$, a contradiction.
Therefore $\ga_3\geq 0$ and $z\in M$. Hence $M$ is seminormal. 
It is easy to see that $M\in \CC(\Phi)$. By
(\ref{3t5}), {\it Serre dim} $R[t_1^2,t_2^2,t_3^2,t_1t_3,t_2t_3]\leq
  d$.

% \item %{\it Case 2:} 
% If $n=2$, $r\geq 3$ and at least one of $i_j\geq 2$, say
%  $i_1$, then $M$ is not seminormal.  For this, let
%  $z=(t_r^2)^{-1}(t_1t_r^{i_1})(t_2t_r^{i_2})\in K(M)\setminus M$. Then
%  $z^2\in M$ and $z^3=(t_r^2)^{-3}(t_1^3t_r^{3i_1})(t_2^3t_r^{3i_2})$ $=
%  t_1^2t_2^2(t_1t_r^{i_1})(t_2t_r^{i_2}) (t_r^2)^{i_1+i_2-3}\in M$.
% 
% \item %{\it Case 3:} 
% If $n=2$, $r\geq 4$ and $i_j=1$ for all $j$, then 
% $M$ is not seminormal.  For this, let
% $$z=(t_r^2)^{-1}(t_1t_r)(t_2t_r)(t_3t_r)\in K(M)\setminus M.$$ Then
% $z^2\in M$ and $z^3=(t_1^2t_2^2t_3^2)(t_1t_r)(t_2t_r)(t_3t_r)\in M$.
% 
% \item %{\it Case 4:} 
% If $r,n\geq 3$ and $i_j\geq 1$ for all $j$, then $M$ is not
% seminormal. For this, let $$z=
% t_r^{-n}(t_1t_r^{i_1})^{n-1}(t_2t_r^{i_2})^{n-1}=(t_1t_2)^{n-1}
% t_r^{(i_1+i_2-1)n-(i_1+i_2)}\in K(M)-M.$$ Then
% $$z^2=(t_1t_2)^{n}(t_1t_r^{i_1})^{n-2}(t_2t_r^{i_2})^{n-2}t_r^{(i_1+i_2-2)n}\in
% M,$$
% $$z^3=(t_1t_2)^{2n}(t_1t_r^{i_1})^{n-3}(t_2t_r^{i_2})^{n-3}t_r^{n(2i_1+2i_2-3)}
% \in M. $$
%  \end{enumerate}
  $\hfill \gj$
\end{example}

\begin{remark}\label{3r11}
(1) Let $R$ be a ring and $P$ a projective $R$-module of rank $\geq
2$. Let $\ol R$ be the seminormalization of $R$. It follows from
arguments in Bhatwadekar (\cite{Bh89}, Lemma 3.1) that $P\ot_R \ol R$ has a
unimodular element if and only if $P$ has a unimodular element. 

(2) Assume $R$ is a ring of dimension $d$ and $M\in \CC(\Phi)$. Let $\ol M$ be
the
seminormalization of $M$. If $\ol M$ is in $C(\Phi)$, then {\it Serre dim} $R[M]
\leq max\{1,d\}$, using (\cite{Bh89} and \ref{3t5}).

%(3) Assume $R$ is a ring of dimension $d$ and $M\subset \BZ_+^2$ a
%finitely generated rank $2$ monoid. Then {\it Serre dim} $R[M]
%\leq max\{1,d\}$ using (\ref{3r5}).

(3) Let $(R,\mm,K)$ be a regular local ring of dimension $d$ containing a
field $k$ such that either char $k=0$ or char $k=p$ and tr-deg
$K/\BF_p\geq 1$. Let $M$ be a seminormal monoid. Then, using Popescu (\cite{Po}, Theorem 1) and Swan
(\cite{Swan}, Theorem 1.2), we get {\it Serre dim} $R[M]=0$. If
$M$  is not seminormal, then {\it Serre dim} $R[M]=1$ using
(\cite{G1}, \cite{Bh89} and \cite{Swan}).
 \end{remark}

%%%%%%%%%%%%%%%%%%%%%%%%%%%%%%%%%%%%%%%%%%%%%%%%
%\subsection{Examples of monoids in $C(\Phi)$}

\begin{example}\label{3e1}
 For a monoid $M$, $\ol M$ denotes the seminormalization of $M$.
\begin{enumerate}
\item 
 Let $M\subset \BZ_+^2$ be a $\Phi$-simplicial monoid generated by $t_1^n, t_1t_2, t_2^n$,
 where $n\in \BN$. We claim that $M$ is normal. To see this, let
$z=t_1^it_2^j=(t_1^n)^p(t_1t_2)^q(t_2^n)^r\in K(M)$ with $p,q,r\in \BZ$ such that
$z^t\in M$ for some $t>0$. Then $i,j\geq 0$. We need to show that
$z\in M$.  We may assume that $0\leq q<n$. Since $i,j\geq 0$, we get
$p,r\geq 0$.  Thus $z\in M$ and $M$ is normal.  Hence, by (\ref{3c11}),
{\it Serre dim} $R[t_1^n,t_1t_2,t_2^n]\leq d$. 

\item 
The monoid $M\subset \BZ_+^2$ generated by $t_1^2, t_1t_2^2, t_2^2$
 is seminormal but not normal.  
 For this, let $z=(t_1t_2^2)(t_2^2)^{-1}=t_1 \in K(M)\setminus M$. Then
$z^2\in M$ showing that $M$ is not normal. For seminormality, let
$z=(t_1^2)^{\ga}(t_1t_2^2)^\gb(t_2^2)^\gm\in K(M)$ with $\ga,\gb,\gm\in \BZ$
be such that $z^2,z^3\in M$. We may assume $0\leq \gb\leq 1$.  If
$\gb=0$, then $\ga,\gm\geq 0$ and hence $z\in M$. If $\gb=1$, then
$z^2\in M$ implies $\ga\geq 0$ and $\gm +1\geq 0$. If $\gm=-1$, then
$z^3=(t_1)^{6\ga+3}\notin M$, a contradiction. Hence $\gm\geq 0$,
proving that $z\in M$ and $M$ is seminormal.
It is easy to see that $M\in \CC(\Phi)$.
Therefore, by
(\ref{3t5}), {\it Serre dim} $R[t_1^2,t_1t_2^2,t_2^2]\leq d$.

\item
Let $M$ be a monoid generated by $(t_1^2,t_1t_2^j, t_2^2)$, where
$j\geq 3$. Then $M$ is not seminormal.  For this, if
$z=(t_1t_2^j)(t_2^2)^{-1}=t_1t_2^{j-2}\in K(M)\setminus M$, then
$z^2=t_1^2t_2^{2(j-2)}$ and $z^3=(t_1^2)(t_1t_2^j)(t_2^{2j-6})$ are in
$M$, showing that $M$ is not seminormal.
 
 If $j=3$, then observe that $t_1t_2$ belongs to $\ol M$. Since the
 monoid generated by $t_1^2,t_1t_2,t_2^2$ is normal, we get that $\ol
 M$ is generated by $t_1^2,t_1t_2,t_2^2$. Hence {\it Serre dim} $R[\ol
   M]\leq d$ by $(1)$ above.
 
 Observe that if $j$ is odd, then $\ol M=(t_1^2,t_1t_2,t_2^2)$ and if $j$ is even, then $\ol M=(t_1^2,t_1t_2^2,t_2^2)$.
 So {\it Serre dim} $R[\ol M]\leq d$ by $(1,2)$ above.
 
 In both cases, applying (\ref{3r11}(1)), we get {\it Serre dim}
 $R[M]\leq$ max $\{1,d\}$.
 
\item %{\it Case 4:} 
Let $M$ be a monoid generated by $(t_1^3,t_1t_2^2,t_2^3)$
 Then $M$ is not seminormal. For
this, let $z=(t_1t_2^2)^{2}t_2^{-3}\in K(M)\setminus M$. Then
$z^2=t_1^3(t_1t_2^2) \in M$ and
$z^3=t_1^6t_2^{3}\in M$. Hence seminormalization of $M$ is $\ol M=(t_1^3,t_1^2t_2,t_1t_2^2,t_2^3)$.
By (\ref{3e4}), {\it Serre dim} $R[\ol M]\leq d$. Therefore, applying (\ref{3r11}(1)), 
we get {\it Serre dim} $R[M]\leq$ max $\{1,d\}$.
\end{enumerate}
$\hfill \gj$
\end{example}

%%%%%%%%%%%%%%%%%%%%%%%%%%%%%%%%%%%%%%%%%%%%%%%%%%%%%%%%%%%%%%%%%%%%%%%%

\subsection{Monoid algebras over $1$-dimensional rings}

The following result proves	 (\ref{1a}(i)).

\begin{theorem}\label{3t15}
 Let $R$ be a ring of dimension $1$ and $M$ a $c$-divisible monoid. If
 $P$ is a projective $R[M]$-module of rank $r\geq 3$, then $P\cong \wedge^rP\op R[M]^{r-1}$.
\end{theorem}

\begin{proof}
If $R$ is normal, then we are done by Swan \cite{Swan}. Assume $R$ is
not normal.  

{\it Case 1.} Assume $R$ has finite normalization.  Let
$\ol R$ be the normalization of $R$ and $C$ the conductor ideal of the
extension $R\subset \ol R$. Then $\hh C=1$.  Hence $R/C$ and
$\ol R/C$ are zero dimensional rings.  Consider the following fiber
product diagram
\[
\xymatrix{
R[M] \ar@{->}[r]
     \ar@{->}[d]
& \ol R[M] 
     \ar@{->}[d]
\\
(R/C)[M] \ar@{->}[r]
     & (\ol R/C)[M]      
}
\]
If $P'=\wedge^rP\op R[M]^{r-1}$, then by Swan \cite{Swan}, $P\ot \ol
R[M]\cong \wedge^r(P\ot \ol R[M])\op \ol R[M]^{r-1} \cong P'\ot \ol
R[M]$.  By Gubeladze \cite{Gu}, $P/CP$ and $P'/CP'$ are free
$(R/C)[M]$-modules.  Further, $\SL_r((\ol R/C) [M])=\E_r((\ol
R/C)[M])$ for $r\geq3$, by Gubeladze \cite{Gu90}.  Now using
standard arguments of fiber product diagram, we get $P\cong
P'$.

{\it Case 2.} Now $R$ need not have finite normalization. 
 We may assume $R$ is a reduced ring with connected spectrum. Let
 $S$ be the set of all non-zerodivisors of $R$. 
 By \cite{Gu}, $S^{-1}P$ is a free $S^{-1}R[M]$-module. 
Choose $s\in S$
 such that $P_s$ is a free $R_s[M]$-module.
 
  Now we follow the arguments of Roy (\cite{Roy82},Theorem 4.1). Let $\hat R$
 denote the $s$-adic completion $R$. Then $\hat R_{red}$ has a finite normalization. 
 Consider the following fiber product diagram
 \[
\xymatrix{
R[M] \ar@{->}[r]
     \ar@{->}[d]
&\hat R[M] 
     \ar@{->}[d]
\\
R_s[M] \ar@{->}[r]
     &\hat R_s[M]      
}
\]
Since $\hat R_s$ is a zero dimensional ring, by \cite{Gu90},
$\SL_r(\hat R_s[M])=\E_r(\hat R_s[M])$ for $r\geq3$. If
$P'=\wedge^rP\op R[M]^{r-1}$, then $P_s$ and $P'_s$ are free
$R_s[M]$-modules and by Case 1, $P\ot \hat R[M]\cong P'\ot \hat R[M]$.
By (\ref{3p8}), $P\cong P'$. This completes the proof. $\hfill \gj$
\end{proof}

% % % % % % % % % % % % % % % % % % % % % % % % % % % % % % % % % % % % %
%\section{Unibranched ring and arbitrary monoids } 

The following result is due to Kang (\cite{K}, Lemma 7.1 and Remark).

\begin{lemma}\label{3l18}
Let $R$ be a $1$-dimensional uni-branched affine algebra over an
algebraically closed field, $\ol R$ the normalization of $R$ and $C$
the conductor ideal of the extension $R\subset \ol R$.  Then $\ol
R/C=R/C+a_1R/C+\cdots+a_mR/C$, where $a_i\in \sqrt{C}$ the radical
ideal of $C$ in $\ol R$.
\end{lemma}

\begin{lemma}\label{3t17}
Let $R$ be a $1$-dimensional ring, $\ol R$ the normalization of $R$ and $C$ the
conductor ideal of the extension $R\subset \ol R$.  Assume $\ol
R/C=R/C+a_1R/C+\cdots+a_mR/C$, where $a_i\in \sqrt{C}$ the radical
ideal of $C$ in $\ol R$.  Let $M$ be a monoid and write $A=\ol
R/C$. 

(i) If $\gs\in \SL_n(A[M])$, then $\gs=\gs_1\gs_2$, where
$\gs_1\in \SL_n((R/C)[M])$ and $\gs_2\in \E_n(A[M])$.

(ii) If $P$ is a projective $R[M]$-module of rank $r$, then $P\cong
\wedge^rP\op R[M]^{r-1}$.
\end{lemma}

\begin{proof}
(i) Let $\gs=(b_{ij})\in \SL_n(A[M])$.  Write
$b_{ij}=(b_{ij})_0+(b_{ij})_1a_1+\cdots+(b_{ij})_ma_m$, where
$(b_{ij})_l \in (R/C)[M]$. If $\ga=((b_{ij})_0)$, then
$det(\gs)=1=det(\ga)+c$, where $c\in (\sqrt{C}/C)[M]$. Since $c\in
(R/C)[M]$ is nilpotent, $det(\ga)$ is a unit in $(R/C)[M]$. Let $\gb=$
diagonal $(1/(1-c),1,\ldots,1) \in \GL_n((R/C)[M])$ and $\gs_1=\ga \gb
\in \SL_n((R/C)[M])$.

Note that $\gs_1^{-1}\gs=\gb^{-1} \ga^{-1}\gs=\gb^{-1}\, 1/(1-c)\, \ol \ga \gs$,
where $\ol \ga=((\ol b_{ij})_0)$, $(\ol b_{ij})_0$ are minors of $ (b_{ij})_0$.
$$\gs_{2}:=\gs_1^{-1}\gs=\left[
\begin{array}{cccc}
1 & 0&\cdots & 0\\
0 & \frac{1}{1-c}& \cdots & 0\\
\vdots&\vdots&\cdots &\vdots\\
0 & 0 &\cdots & \frac{1}{1-c}
\end{array}\right]
\left[
\begin{array}{cccc}
1+c_{11} & c_{12}&\cdots & c_{1n}\\
c_{21} & 1+c_{22}& \cdots & c_{2n}\\
\vdots&\vdots&\cdots &\vdots\\
c_{n1} & c_{n2} &\cdots & 1+c_{nn}

\end{array}\right], $$
where $c_{ij}\in (\sqrt{C}/C)[M]$.

Note that $\gs_2\in \SL_n(A[M])$ and $\gs_2=Id$ modulo the nilpotent
ideal of $A[M]$.  Hence $\gs_2\in \E_n(A[M])$. Thus we get
$\gs=\gs_1\gs_2$ with the desired properties. 

(ii) Follow the proof of (\ref{3t15}) and use (\ref{3t17}(i)) to get the
result. 
$\hfill \gj$
\end{proof}

Now we prove (\ref{1a}(ii))
which follows from (\ref{3l18}) and (\ref{3t17}).

\begin{theorem}\label{3t16}
Let $R$ be a $1$-dimensional uni-branched affine algebra over an
algebraically closed field and $M$ a monoid. If $P$ is a projective
$R[M]$-module of rank $r$, then $P\cong \wedge^rP\op R[M]^{r-1}$.
\end{theorem}

%%%%%%%%%%%%%%%%%%%%%%%%%%%%%%%%%%%%%%%%%%%%%%%%%

\section{Applications}
Let $R$ be a ring of dimension $d$ and $Q$ a finitely generated $R$-module.
Let $\mu(Q)$ denote the  minimum number of generators of
$Q$. By Forster \cite{F67} and Swan \cite{Swan67},
$\mu(Q)\leq max\{\mu(Q_{\p})+\dim (R/\p) | \p \in
\Spec (R), Q_{\p}\neq 0\}$.  In particular, if $P$ is a projective
$R$-module of rank $r$, then $\mu(P)\leq r+d$. 

The following result is well known.

\begin{theorem}\label{01}
 Let $A$ be a ring such that {\it Serre dim} $A\leq d$.  Assume $A^m$
 is cancellative for $m\geq d+1$.  If $P$ is a projective $A$-module
 of rank $r\geq d+1$, then $\mu (P) \leq r+d$.
\end{theorem}

\begin{proof}
Assume $\mu(P)=n>r+d$.  Consider a surjection $\phi: A^n\surj P$ with
$Q=ker(\phi)$.  Then $A^n\cong P\op Q$. Since $Q$ is a projective
$A$-module of rank $\geq d+1$, $Q$ has a unimodular element $q$.
Since $\phi(q)=0$, $\phi$ induces a surjection $\ol \phi:A^n/qA^n\surj
P$. Since $n-1>d$, $A^{n-1}$ is cancellative. Hence $A^{n-1}\cong
A^n/qA$ and $P$ is generated by $n-1$ elements, a contradiction.
$\hfill \gj$
\end{proof}

The following result is immediate from (\ref{01}, \ref{3t5}, \ref{3c11} and 
\cite{DK}).

\begin{corollary}
 Let $R$ be a ring of dimension $d$, $M$ a monoid and $P$ a 
 projective $R[M]$-module of rank $r>d$. Then:
 
 $(i)$ If $M\in \CC(\Phi)$, then $\mu(P)\leq r+d$.
 
 $(ii)$ If $M\subset \BZ^2_+$ is a normal monoid of rank $2$, then $\mu(P)\leq r+d$.
\end{corollary}

Let $M$ be a $c$-divisible monoid, $R$ a ring of dimension $d$ and
$n\geq max \{2,d+1\}$. Then Schaubh\"{u}ser \cite{Sch} proved that
$E_{n+1}(R[M])$ acts transitively on $\Um_{n+1}(R[M])$.  Using
Schaubh\"{u}ser's result and arguments of Dhorajia-Keshari (\cite{DK},
Theorem 4.4), we get that if $P$ is a projective $R[M]$-module of rank
$n$, then $E(R[M]\op P)$ acts transitively on $\Um(R[M]\op P)$.
Therefore the following result is immediate from (\ref{01} and
\ref{3t15}).

\begin{corollary}
 Let $R$ be a ring of dimension $1$, $M$ a $c$-divisible monoid and $P$ a 
 projective $R[M]$-module of rank $r\geq3$. Then $\mu(P)\leq r+1$.
\end{corollary}

\medspace
\medskip
%%%%%%%%%%%%%%%%%%%%%%%%%%%%%%%%%%%%%%%%%%%%%%%%%%%%
\noindent{\bf Acknowledgement.} We would like to thank the referee for
his/her critical remark. The second author would like to thank 
C.S.I.R., India for their fellowship.

%%%%%%%%%%%%%%%%%%%%%%%%%%%%%%%%%%%%%%%%%%%%%%%%%%%%%%%%%%%%
{\small
{}

}

\end{document}